\documentclass{gen-j-l}

\newtheorem{theorem}{Theorem}[section]
\newtheorem{lemma}[theorem]{Lemma}
\newtheorem{prop}[theorem]{Proposition}

\theoremstyle{definition}

\theoremstyle{remark}

\numberwithin{equation}{section}

\usepackage{amssymb}
\usepackage{epsfig}
\usepackage{graphicx}

\newcommand{\BR}{{\mathbb R}}
\newcommand{\conv}{\mathrm{conv}}
\newcommand{\vol}{\mathrm{vol}}

\begin{document}

\title{Recovering Lexicographic Triangulations}


\author{Carl W. Lee}
\address{University of Kentucky}
\curraddr{}
\email{lee@uky.edu}
\thanks{}

\author{Wendy Weber}
\address{Central College, Pella, IA}
\curraddr{}
\email{weberw@central.edu}
\thanks{}


\date{}

\dedicatory{}


\begin{abstract}
Given a finite set $V=\{v^1, \dots, v^n\} \subset \mathbb R^d$ with $\dim(\conv(V))=d$, a triangulation $T$ of $V$ is a collection of distinct subsets $\{T_1, \dots, T_m\}$ where $T_i \subseteq V$ is the vertex set of a $d$-simplex, $\mathrm{conv} (V)=\bigcup_{i=1}^m \mathrm{conv} (T_i)$, and $T_i \cap T_j$ is a common (possibly empty) face of both $T_i$ and $T_j$.  Associated with each triangulation $T$ of $V$ is the GKZ-vector $\phi(T)=(z_1, \dots, z_n)$ where $z_i$ is the sum of the volumes of all $d$-simplices of $T$ having $v^i \in V$ as a vertex.  It is clear that given $V$ and a triangulation $T$ we can find $\phi(T)$.   The focus of this paper is recovering a lexicographic triangulation from its GKZ-vector.  

The motivation for studying triangulations and their GKZ-vectors arises from  the work of Gel'fand, Kapranov, and Zelevinski\v{\i} \cite {G86,GZK90,GZK91,GKZ} in which they illuminate connections between regular triangulations and subdivisions of Newton polytopes, and  generalized discriminants and determinants.

The {\it secondary polytope}, $\Sigma (V)$, of an arbitrary  finite point set $V \subset \mathbb R^d$, introduced by Gel'fand, Kapranov, and Zelevinski\v{\i}, is defined to be the convex hull of the GKZ-vectors of all triangulations of $V$.   They showed  the vertices of $\Sigma (V)$ are in one-to-one correspondence with the {\it regular} triangulations of $V$.  Since the GKZ-vector of a  regular triangulation is uniquely associated with that triangulation, a natural question is how that triangulation can be recovered from its vector.  We answer this question in the case that the associated triangulation is lexicographic. 

\end{abstract}

\maketitle


\begin{sloppypar}

\section{General Definitions }

\label{chapter-pre}

Let $V=\{v^1, \dots, v^{n}\}$ be a finite set of points in ${\mathbb R}^d$ such that the convex hull of $V$, ${\mathrm{conv}}(V)$, is a $d$-dimensional (convex) polytope; we say $\dim ({\mathrm{conv}} (V) )=d$.   We also assume for convenience that no points of $V$ are included with multiplicity greater than one, though it is not difficult to modify the results of this paper to allow this possibility.   For $S\subseteq\BR^d$ we denote by ${\mathrm{aff}}
(S)$ the affine span of $S$, the intersection of all affine sets containing $S$.  For $a \in {\mathbb R}^d$, $a\neq O$, and $\alpha \in {\mathbb R}$, a  {\it hyperplane}, $H$, is a set of the form $H=\{x: a^Tx=\alpha\}$.  We say  {\it $H$ is a supporting hyperplane to the set $S$} if $a^Tx \leq \alpha$, for every $x \in S$, and $H \cap S \neq \emptyset$, $S \nsubseteq H$. Given a finite set $S$ of points, a subset $T\subseteq S$ is a {\it face of $S$} if $T=\emptyset$ or $T=S \cap H$ where $H$ is a supporting hyperplane to $S$.  In this case we call $T$ a {\it facet\/} of $S$ if $\dim(\conv(T))=\dim(\conv(S))-1$, a {\it subfacet\/} of $S$ if $\dim(\conv(T))=\dim(\conv(S))-2$, and a  {\it vertex\/} of $S$ if $T$ is a single point.

A collection of subsets $S=\{S_1, \dots, S_m\}$ of $V$, where ${\mathrm{conv}}(V)$ and ${\mathrm{conv}}(S_i)$, for all $i$, are $d$-dimensional, is a {\it subdivision} of $V$ if (i) ${\mathrm{conv}}(V)=\bigcup_{i=1}^m {\mathrm{conv}}(S_i)$, and  (ii) $S_i \cap S_j$ is a common (possibly empty) face of both $S_i$ and $S_j$ for all $i \neq  j$.  If the only set of $S$ is $V$ itself, then $S$ is the {\it trivial subdivision}.  Suppose $S=\{S_1, \dots, S_m\}$ and $T=\{T_1, \dots, T_n\}$ are subdivisions of the point set $V$.  The subdivision $T$ is a {\it refinement} of the subdivision $S$ if for every $T_i$ there is an $S_j$ such that $T_i \subseteq S_j$.  In this case, we write $T   \leq S$.  Note that by definition, if $T   \leq S$, then every set of $S$ is subdivided by particular sets of $T$.   We say  $T$ is {\it finer} than $S$ and  $S$ is {\it coarser} than $T$.  If $S \neq T$, then $T$ is a {\it proper refinement} of $S$ and we write $T  < S$. Note that every subdivision is a refinement of the trivial subdivision and the trivial subdivision is coarser than every subdivision.  A subdivision $S$ is a  {\it minimal nontrivial subdivision of $V$} if the only subdivision coarser than it is the trivial subdivision.  Suppose  $S=\{S_1, \dots, S_m\}$ is a subdivision of $V$.  Then the point {\it $v$ is present in $S$} if there is an $i$ such that $v \in S_i$. If no such $i$ exists, then we say {\it $v$ is absent from $S$}.  Note that if $v$ is absent from $S$, then $v$ will be absent from every refinement of $S$, and a refinement $S'$ of $S$ may have fewer points present; it certainly cannot have more points present than $S$ does.  

Suppose $\dim(V-\{v\})=d-1$.  Then we call the set $V$ a {\it pyramid\/} with apex $v$ and base $V-\{v\}$.  In this case it is easy to see that any subdivision $\{S_1,\ldots, S_m\}$ of $V$ is a set of pyramids of the form $\{T_1\cup\{v\},\ldots, T_m\cup\{v\}\}$, where $\{T_1,\dots,T_m\}$ is a subdivision of the base.  

If $S=\{S_1, \dots, S_m\}$ is a subdivision of $V \subset {\mathbb R}^d$, ${\mathrm{conv}}(V)$ is $d$-dimensional, and each $S_i$ has exactly $d+1$ points, the
subdivision $S$ is a {\it triangulation}.  This definition implies each of the $d$-polytopes ${\mathrm{conv}}(S_1), \dots, {\mathrm{conv}}(S_m)$ of the subdivision $S$ is a $d$-simplex.



\section{Lexicographic Subdivisions and Triangulations of Finite Point Sets}
\label{section-pre-triangulation}

Suppose $V \subset {\mathbb R}^d$ is a finite point set with ${\mathrm{conv}}(V)$ a $d$-polytope.  Let $F$ be a facet of $V$ and let $v$ be a point in ${\mathbb R}^d$.  Since $F$ is a facet of $V$, there is a unique hyperplane $H$ containing $F$.  The polytope ${\mathrm{conv}}(V)$ is contained in precisely one of the closed half spaces determined by $H$.  If $v$ is contained in the opposite open half space, then $F$ is said to be a {\it facet of $V$ visible from $v$}.  

If $V$ is a pyramid with apex $v$ and $S$ is a subdivision of $V$, or if $\dim(\conv(V-\{v\}))=d$ and $S$ is a refinement of the subdivision 
$$\{ V -\{v\}\}\cup \{F\cup \{v\}: \text{ $F$ is a facet of   $V -\{v\}$ visible from $v$} \},$$
 for some vertex $v$ of $V$, then we say  $S$ is a {\it generalized ear subdivision} of $V$.   In this case, we say  $v$ is an {\it ear point} of $S$; the collection of pyramids $E(S,v):= \{S_i :v \in S_i\}$ is  the {\it ear of $S$ given by $v$}.  Note that if $v \in {\mathrm{conv}} (V-\{v\})$ (i.e., $v$ is not a vertex of $V$), then the subdivision  $S$ is a refinement of  $\{V-\{v\}\}$ where the point $v$ is actually absent from $S$ and thus in this case we can define $E(S,v)$ to be empty.

Suppose $S=\{S_1, \dots, S_m\}$ is a subdivision of $V \subset {\mathbb R}^d$, $\dim ( {\mathrm{conv}} (V))=d$, and $v$ is present in $S$.  The result of {\it pulling the point $v$} is the subdivision $S'$ of $V$ obtained by modifying each $S_i \in S$ by the following:  (i) If $v \notin S_i$, then $S_i \in S'$. (ii)  If $v \in S_i$, then for every facet $F$ of $S_i$ not containing $v$, $F \cup \{v\} \in S'$.  The result of {\it pushing the point $v$} is the subdivision $S'$ of $V$ obtained by modifying each $S_i \in S$ by the following:  (i) If $v \notin S_i$, then $S_i \in S'$. (ii) If $v \in S_i$ and ${\mathrm{conv}} (S_i -\{v\})$ is $(d-1)$-dimensional (i.e., $S_i$ is the point  set of a pyramid with apex $v$), then $S_i \in S'$.  (iii) If $v \in S_i$ and ${\mathrm{conv}} (S_i -\{v\})$ is $d$-dimensional, then $S_i -\{v\} \in S'$.  Also, if $F$ is any facet of $S_i-\{v\}$ visible from $v$, then $F \cup \{v\} \in S'$.
Note that in both definitions the subdivision $S'$ is a refinement of $S$.  

If we start with the trivial subdivision and push a vertex $v$ of $V$, the subdivision $S$ we obtain will be a generalized ear subdivision with $v$ as an ear point of $S$.  Further, in every subdivision of $S$, the collection of sets containing $v$ will be a subdivision of $E(s,v)$.

If we start with the trivial subdivision of $V$ and push a non-vertex of $V$, the resulting subdivision will be $S=\{V-\{v\}\}$.  Hence, $v$ will be absent from every refinement of $S$.  

If we start with the trivial subdivision of $V$ and pull a point  $v \in V$ to obtain the subdivision $S$, then $v$ will be in every  $d$-dimensional set of every refinement of $S$.

It is important to note that if $S$ is a subdivision of $V$ and $V$ is a pyramid with apex $v$, then pulling or pushing $v$ will leave the subdivision unchanged.
 
Any subdivision $S$ of $V$ constructed by starting with the trivial subdivision of $V$ and pulling and/or pushing some/all the points of $V$ in some order is a {\it lexicographic subdivision} of $V$.  Pulling and/or pushing all of the points in $V$ in some order yields a lexicographic triangulation.

%
%


\section{Regular Subdivisions and Triangulations}
\label{section-pre-reg}

A subdivision $S$ of a finite point set $V \subset {\mathbb R}^d$ with $Q={\mathrm{conv}}(V)$ $d$-dimensional is {\it regular} if it arises from a polytope $P\subset {\mathbb R}^{d+1}$ in the following way:

\begin{enumerate}
\item $\pi(P)=Q$ via the projection $\pi:{\mathbb R}^{d+1} \rightarrow {\mathbb R}^d$ which deletes the last  coordinate, and 
\item the subdivision $S$ is given by the point sets of the  lower (or upper) facets of $P$ projected down  to $Q$.
\end{enumerate}

Here, a lower (respectively, upper) facet of $P$ is one with outer normal vector $(u_1,\ldots,u_{d+1})$ with $u_{d+1}<0$ (respectively, $u_{d+1}>0$).

Note that subdivisions of $V \subset {\mathbb R}^d$ obtained by lifting exactly one vertex (by a positive amount), taking the convex hull, and then projecting the lower facets, will give a generalized ear subdivision.   If $v$ is a non-vertex, then the subdivision obtained will be $S=\{ V-\{v\} \}$.  Thus, such a generalized ear subdivision is regular.  It is not hard to show that this particular subdivision is a minimally nontrivial one.

In fact, all lexicographic triangulations (and subdivisions) are regular.  In particular, if $v^1, \dots, v^n$ are pulled/pushed in that order, then the corresponding triangulation is obtained by choosing $P=\mathrm{conv}\{(v^1,\varepsilon_1),\ldots,(v^n,\varepsilon_n)\}$ with $|\varepsilon_1| \gg |\varepsilon_2| \gg \cdots \gg |\varepsilon_n| \gg 0$ where $\varepsilon_i >0$ if $v^i$ is pushed and $\varepsilon_i <0$ if $v^i$ is pulled \cite{LEE91}. 

Associated with each triangulation $T$ of $V$ is the {\it GKZ-vector\/} $\phi(T):=(z_1(T), \dots, z_n(T))$ where $z_i(T)$ is the sum of the volumes of all $d$-simplices of $T$ having $v^i \in V$ as a vertex.  The {\it secondary polytope\/} of $V$ is $\Sigma(V):=\conv(\{\phi(T):T$ is a triangulation of $V\})$.  In general, two different triangulations may have the same GKZ-vector.  However, if $T$ is a regular triangulation and $T'$ is any other triangulation, then $\phi(T)\not=\phi(T')$.  Thus regular triangulations are uniquely determined by their GKZ-vectors.  It is known that the vertices of $\Sigma(V)$ are precisely the GKZ-vectors of the regular triangulations \cite{GZK90,GZK91}.

For more details of subdivisions, triangulations, and secondary polytopes, we refer the reader to \cite{DRS10,ZIEGLER}.


\section{Comments on Ear Points}
\label{section-ear-d-diml}

In this section we discuss ear points of a subdivision $S$ of $V \subset {\mathbb R}^d$.  The first  result of this section gives an equivalent condition for ear points of triangulations; the second shows  the subdivision obtained by pushing $v$ in the trivial subdivision is a minimal nontrivial {\it regular} subdivision.

Suppose $w\in V \subset {\mathbb R}^d$ and $\conv(V-\{w\}))$ has dimension $d$.
Then the {\it shadow of $w$ in $V -\{w\}$} is the set of facets of ${\mathrm{conv}}(V-\{w\})$ visible from $w$ together with all of their faces.
Consider all subfacets of $V -\{w\}$ that are contained in exactly one facet in the shadow of $w$ in $V -\{w\}$.  The {\it shadow boundary of $w$ in $V$}, denoted $\partial(V-\{w\},w)$, is the set of all such subfacets, together with all their faces.

The proofs of the following propositions are straightforward.

\begin{prop}
Suppose $V \subset {\mathbb R}^d$ is a finite point set,  $T$ is a triangulation of $V$, and $\dim ({\mathrm{conv}} (V))=d$. Then, $w$ is an ear point of $T$ if and only if $w$ is an extreme point of $V$ and ${z}_w(T)={\mathrm{vol}}( V)- {\mathrm{vol}} (V-\{w\}).$ 
\label{thm-earpoints}
\end{prop}

\begin{prop}
Suppose $V \subset {\mathbb R}^d$, $\dim ({\mathrm{conv}} (V))=d$, $w \notin {\mathrm{conv}} (V-\{w\})$, and  $\dim({\mathrm{conv}} (V - \{w\}))=d$.  If $S$ is the subdivision of $V$ resulting from pushing $w$ in the trivial subdivision, then $S$ is a minimal nontrivial regular subdivision of $V$.
\end{prop}

\section{Pulling and Pushing in Lexicographic Triangulations}

\label{chapter-BIG}

The remainder of the paper is devoted to demonstrating how one can use a form of greedy algorithm to recover a lexicographic triangulation $T$ of a finite set $V \subset {\mathbb R}^d$ from $\phi(T)$.  The triangulation will be assumed to be lexicographic, but the ordering of the vertices and which are pushed or pulled is not given.

We begin with some lemmas that lead to some useful definitions

\begin{lemma}
Let $W\subseteq V \subset {\mathbb R}^d$ with $\dim({\mathrm{conv}}(W))=d$.  
If $T$ is a refinement of the subdivision $S$ of $W$ obtained by pulling $v^i$, then $z_i(T)=z^{\max}_i(W)$, where
\[
z^{\max}_i(W):=
\left\{
\begin{array}{ll}
{\mathrm{vol}}(\conv(W)), & v^i\in W,\\
0, & v^i\not\in W.
\end{array}
\right.
\]
\end{lemma}

\begin{proof}
From earlier observations, if $v^i\in W$ then $S$ consists of a collection of pyramids each having $v^i$ as a vertex, and any refinement $T$ of $S$ preserves this property.  Thus $v^i$ is in every $d$-simplex of the triangulation.  
\end{proof}

We also have the following easy result.

\begin{lemma}
Let $W\subseteq V \subset {\mathbb R}^d$ with $\dim({\mathrm{conv}}(W))=d$.  
Then $\max\{z_i(T):T$ is a triangulation of $W\}=z^{\max}_i(W)$, and if in any triangulation $T$ of $W$ we have $z_i(T)=z^{\max}_i(W)$, then $v^i$ must be in every $d$-simplex of $T$.
\end{lemma}

\begin{lemma}
Let $W\subseteq V \subset {\mathbb R}^d$ with $\dim({\mathrm{conv}}(W))=d$.  
If $T$ is a refinement of the subdivision $S$ of $W$ obtained by pushing $v^i$, then $z_i(T)=z^{\min}_i(W)$, where
\[
z^{\min}_i(W):=
\left\{
\begin{array}{ll}
{\mathrm{vol}}(\conv(W))-{\mathrm{vol}}(\conv(W-\{v^i\}), & v^i\in W,\\
0, & v^i\not\in W.
\end{array}
\right.
\]
\end{lemma}

\begin{proof}
From earlier observations, the $d$-polytopes containing $v^i$ in $S$ together are a set of pyramids subdividing $E(S,v^i)$, and any refinement $T$ of $S$ preserves this property.  Thus 
\[
z_i(T)=\sum_{T_j\in E(T,v^i)}\vol(T_j)={\mathrm{vol}}(\conv(W))-{\mathrm{vol}}(\conv(W-\{v^i\}).
\]
\end{proof}

\begin{lemma}
Let $W\subseteq V \subset {\mathbb R}^d$ with $\dim({\mathrm{conv}}(W))=d$.  
Then $\min\{z_i(T):T$ is a triangulation of $W\}=z^{\min}_i(W)$, and if in any triangulation $T$ we have $z_i(T)=z^{\min}_i(W)>0$, then $T$ must be a generalized ear subdivision of $W$ with ear point $v^i$.  If $z_i(T)=z^{\min}_i(W)=0$ then $v^i$ is not a vertex of $V$ and not in any $d$-simplex of $T$.
\end{lemma}

\begin{proof}
In any triangulation $T$, the sum of the volumes of the $d$-simplices not containing $v^i$ cannot exceed $\vol(\conv(W-\{v^i\}))$.  Thus $z_i(T)\geq z^{\min}_i(W)$, and the only way to achieve this value is for all $d$-simplices containing $v^i$ to be contained in (the closure of) $\conv(W)-\conv(W-\{v^i\})$.
\end{proof}

The above lemmas suggest the following definitions.
Let $T$ be a triangulation of $W\subseteq V \subset {\mathbb R}^d$, $\dim ({\mathrm{conv}} (W))=d$.  Then a point $v^i \in W$ is a {\it candidate to be pushed first in $T$} if $z_i(T)=z^{\min}_i(W)$
and a {\it candidate to be pulled first in $T$} if 
$z_i(T)=z^{\max}_i(W).$


Next we present a few lemmas that  will be used to prove the main theorem.  The proofs follow directly from the definition of a lexicographic triangulation. 

\begin{lemma}
Suppose  $V \subset {\mathbb R}^d, \dim ({\mathrm{conv}}(V) )=d$.  Let $S$ be a subdivision of $V$.  Suppose  $T$ is a lexicographic triangulation obtained by refining $S$ by pulling and/or pushing the points   of $S$ via the  order $v^1, \dots, v^n$.  Then each set $S_i \in S$ is triangulated by pulling and/or pushing the points   of $S_i$ in the induced order.
\label{lemma-L1}
\end{lemma}

Lemma ~\ref{lemma-L1} is seen to be true by recalling that the subdivision $S$ is refined by looking at the sets $S_i \in  S$.  So, pulling or pushing  a point  of $V$ subdivides each set $S_i$ according to the pulling and pushing rules.

The following lemma follows directly from the definition of a lexicographic triangulation.

\begin{lemma}
Suppose  $V \subset {\mathbb R}^d, \dim ({\mathrm{conv}}(V) )=d$.  Let $T$ be a lexicographic triangulation of $V$ and $F$ be a face of $V$.  Then the triangulation of $F$ induced by $T$ is identical to the triangulation of $F$ given by pulling or pushing the points   of $F$ according to  the induced  ordering.
\label{lemma-L2}
\end{lemma}

The next lemma will be used extensively in the proof to the main theorem which uses induction on the number of points.   We will often use the fact that we have a pyramid and can triangulate it according to the triangulation of the base.  The order for the  triangulation of the base, in the main theorem, is obtained using the inductive hypothesis.  Its proof follows readily from the definitions of pushing and pulling.
 
\begin{lemma}
Suppose  $V \subset {\mathbb R}^d, \dim ({\mathrm{conv}}(V) )=d$.  Suppose  $V$ is a pyramid with apex $v$ and $T$ is a lexicographic triangulation of $V$.  Then $T$ can be obtained by pulling or pushing $v$ at any point in the  order.  In fact, $T$ is given by pyramids over the $(d-1)$-simplices of the induced lexicographic triangulation of the base of $V$.
\label{lemma-L3}
\end{lemma}


\section{The Main Theorem}

We now state the main theorem.

\begin{theorem}
\label{thm-BIG}
Let $V=\{v^1, \dots, v^n\} \subseteq {\mathbb R}^d$, $\dim ({\mathrm{conv}} (V))=d$,  and $T$ be a lexicographic triangulation of $V$ given by the order $v^1, \dots, v^n$.  If  $v^k$ is a candidate to be pulled first (or a candidate to be pushed first),  then $T$ can be obtained via the order $v^k, v^1, \dots, v^{k-1}, v^{k+1}, \dots, v^n$, where $v^i$, $i \neq k$,  is pulled (pushed) if it was pulled (pushed) in the original order and $v^k$ is pulled (pushed) if it was a candidate to be pulled (pushed) first in $T$.
\end{theorem}


The proof will be obtained through a sequence of results spread over the next few sections.
We use induction on $|V|$; the base case is straightforward.  But first we handle some easy cases.

\begin{lemma}
The result of Theorem~\ref{thm-BIG} holds if $v^k$ is both a candidate to be pulled first and a candidate to be pushed first in $T$.  
\end{lemma}

\begin{proof}
In this case $z^{\max}_k(V)=z^{\min}_k(V)$, so $\vol(\conv(V-\{v\}))=0$.  Thus $V$ is a pyramid with apex $v^k$, and $v^k$ can be pushed or pulled in any order.  In particular, it can be moved to the first position in the order and either pulled or pushed first.
\end{proof}

\begin{lemma}
The result of Theorem~\ref{thm-BIG} holds if pulling/pushing $v^1$ results in the trivial subdivision of $V$. 
\end{lemma}

\begin{proof} 
In this case $v^1$ is the apex of a pyramid with  base  $V-\{v^1\}$.  By Lemma ~\ref{lemma-L2}, the base can be triangulated via the induced order $v^2, \dots, v^n$. Let $T'$ be the  triangulation of the base given by this order. We claim  $v^k$ is a candidate to be pulled (pushed) first in the induced triangulation $T'$ of the base $V-\{v^1\}$.  This is true since the $d$-volumes of the $d$-simplices of $T$ are proportional to the
$(d-1)$-volumes of their respective base $(d-1)$-simplices of $T'$, so $z_k(T)$ is proportional to $z_k(T')$ for $k\not=1$ with the same constant of proportionality.  We note use  that the $d$-simplices in $E(T,v^k)$ are pyramids with apex $v^1$ over the $(d-1)$-simplices in $E(T',v^k)$. 
Thus, if $v^k$ is a candidate to be pulled (respectively, pushed) first in $T$ then it is a candidate to be pulled (pushed) first in the triangulation $T'$ of the base. 

Applying the inductive hypothesis to the triangulation $T'$ of the base $V-\{v^1\}$ given by the  order $v^2, \dots, v^n$, we obtain $T'$ via the order 
$v^k, v^2, \dots, v^{k-1}, v^{k+1}, \dots, v^n$, where $v^k$ is actually pulled (pushed).  By Lemma \ref{lemma-L3} we have that $T$ is defined  by $T'$ and $v^1$ can be pulled (respectively, pushed) at any point in the order. Thus, we may pull (push) $v^1$ second; $T$ can be obtained via the order $v^k, v^1, v^2, \dots, v^{k-1}, v^{k+1}, \dots, v^n$, where $v^k$ is pulled (pushed) if it was a candidate to be pulled (pushed) first in $T$.  
\end{proof}

From now on we assume that $v^k$ is not a dual candidate and also that pulling/pushing $v^1$ results in a nontrivial subdivision of $V$.
In Section~\ref{second} we show that a candidate to be pulled or pushed first can be moved to the second position.
In Section~\ref{swap} we then prove that the first two vertices in the order can be swapped.

\section{Moving to Second Position}
\label{second}

Lemma~\ref{lemma-part1A} will show  $v^k$ can be moved to the second position in the order when $v^k$ is a candidate to be pulled first.

\begin{lemma}
Assume the hypotheses of Theorem~\ref{thm-BIG}.  Suppose pulling (pushing) $v^1$ gives a nontrivial subdivision of $V$ and $v^k$ is a candidate for
pulling first in $T$.  Then $T$ can be obtained by the order $v^1, v^k, v^2, \dots, v^{k-1}, v^{k+1}, \dots, v^n$, where $v^i, ~i \neq k,$  is pulled (pushed) if it was pulled (pushed) in the original order and $v^k$ is actually pulled.
\label{lemma-part1A}
\end{lemma}

\begin{proof}
Since $v^k$ is a candidate for pulling first, we have $z_k(T)={\mathrm{vol}} (V)$.  Let $ S=\{S_1, \dots, S_m\}$ be the nontrivial subdivision obtained by pulling (pushing) $v^1$ in the trivial subdivision. Since $v^k$ must be a vertex of each $d$-simplex of $T$, we have $v^k \in S_i$, $i=1, \dots, m$.  Now, let $T_i$ be the triangulation of $S_i$, $i=1, \dots, m$, induced by $T$, and $\phi(T_i)$ be the induced GKZ-vector (so that  if $v^j \notin S_i$ then $z_j(T_i)=0$). Then, for each $i$, $$z_k(T_i)  \leq  \max \{z_k(R): \text{$R$ is a triangulation of $S_i$} \} =  {\mathrm{vol}} (S_i).$$  So, $$z_k(T)  = \displaystyle \sum_{i=1}^{m} z_k(T_i)  \leq  \displaystyle \sum_{i=1}^{m} {\mathrm{vol}} (S_i) =  {\mathrm{vol}} (V).  $$
Since $z_k(T)={\mathrm{vol}}(V)$, we must have $z_k(T_i) = {\mathrm{vol}} (S_i)$ for all $i$.  Therefore, $v^k$ is a candidate to be pulled first in each $T_i$. We apply the induction hypothesis to the induced order for each $T_i$.  Thus each $T_i$ can be obtained via the order induced by $v^k, v^1, \dots, v^{k-1}, v^{k+1}, \dots, v^n$ (if $v^j \notin S_i$, then $v^j$ is not in the induced order). 

Recall that the $S_i$ are created  by first pulling/pushing $v^1$.  Thus, to insure we obtain the same $S_i$, we must have $v^1$ first in the order.
If $v^1 \notin S_i$, then pushing or pulling $v^1$ will not affect $T_i$. So, such $T_i$ may be obtained via the order $v^1, v^k, v^2,\dots, v^{k-1}, v^{k+1}, \dots, v^n$ with $v^k$ actually pulled.   Suppose $v^1 \in S_i$.  Since we obtained $S$ by pulling/pushing $v^1$ first, each $S_i$ is a pyramid with apex $v^1$ and base $S_i -\{v^1\}$.  By Lemma~\ref{lemma-L3}, $v^1$ may be pushed or pulled  at any time in the induced order for $T_i$.  Thus, the triangulation $T_i$ of each $S_i$ containing $v^1$  can be obtained via the order $v^1, v^k, v^2,\dots, v^{k-1}, v^{k+1}, \dots, v^n$ with $v^k$ pulled. 

To obtain $T$, we first pull/push $v^1$ and obtain $S$.  We then continue by subdividing each $S_i \in  S$ by the induced order of $T$.  Thus, since each $S_i$ can be triangulated by pulling $v^k$ next (with the rest of the order unchanged), $T$ can be obtained via the order $v^1, v^k, v^2,\dots, v^{k-1}, v^{k+1}, \dots, v^n$.  
\end{proof}

Next we show we can move $v^k$ up to the second position when $v^k$ is a candidate to be pushed first.

\begin{lemma}
Assume the hypothesis of Theorem~\ref{thm-BIG}.  Suppose pulling (pushing) $v^1$ gives a nontrivial subdivision of $V$ and $v^k$ is a candidate for pushing first in $T$.  Then $T$ can be obtained by the order $v^1, v^k, v^2, \dots, v^{k-1}, v^{k+1}, \dots, v^n$, where $v^i, ~i \neq k,$  is pulled (pushed) if it was pulled (pushed) in the original order and $v^k$ is actually pushed.
\label{lemma-part1B}
\end{lemma}

\begin{proof}
Let the nontrivial subdivision obtained by pulling (pushing) $v^1$ be given by $\{S_1, \dots, S_m\}$. Let $T_i$ be the triangulation of $S_i$, $i=1, \dots, m$, induced by $T$, and let $\phi(T_i)$ be the GKZ-vector of $T_i$ induced by the GKZ-vector of $T$ (so that if $v^j \notin S_i$, then $z_j(T_i)=0$).   Recall $z^{\min}_k(S_i):=\min \{z_k(R): \text{$R$ is a triangulation of $S_i$} \}$. Note that we have $z_k(T_i) \geq z^{\min}_k(S_i), ~i=1, \dots, m$.  We claim $z_k(T_i)=z^{\min}_k(S_i), ~i=1, \dots, m$.  

Assume there is a $j$ such that $z_k(T_j) >z^{\min}_k(S_j)$.  Now triangulate each $S_i$, $i=1, \dots , m,$ using the order induced by $v^k, v^2, \dots, v^{k-1}, v^{k+1}, \dots, v^n$ in which all points are pushed.  Since $v^k$ is pushed first in each $S_i$, the triangulation $T'_i$ of $S_i$ has $$z_k(T_i') =\min \{z_k(R):  \text{$R$ is a triangulation of $S_i$} \}=z^{\min}_k(S_i).$$  Let $T'$ be the triangulation of $V$ given by first pulling/pushing $v^1$ and then pushing each point of $v^k, v^2, \dots, v^{k-1}, v^{k+1}, \dots, v^n$.  Then $T'=\bigcup_{i=1}^{m} T_i'$, and $$z_k(T')=\displaystyle \sum_{i=1}^{m}z_k(T_i')=\displaystyle \sum_{i=1}^{m} z^{\min}_k(S_i) <\displaystyle \sum_{i=1}^{m} z_k(T_i)=z_k(T),$$ a contradiction since $v^k$ was a candidate to be pushed first in $T$ and $z_k(T)=\min \{z_k(R):   \text{$R$ is a triangulation of $V$} \}$.  Therefore, $z_k(T_i)=z^{\min}_k(S_i)$ for all $i$, and $v^k$ is a candidate to be pushed first in each $S_i$.  Note that $T_i$ is induced  by the order for $T$ (so that if $v^j \notin S_i,$ pulling (pushing) $v^j$ does not change the given subdivision).  Since $|S_i|
< |V|$ for all $i,$ we apply the induction hypothesis and  obtain each $T_i$ via the order $v^k, v^1, v^2, \dots, v^{k-1}, v^{k+1}, \dots, v^n$, with $v^k$ being pushed.  Since $v^1$ is pulled (pushed) first,  each $S_i$ is a pyramid with apex $v^1$ or it does not contain $v^1$. In either case,  $v^1$ may be pulled (pushed) at any point in the order for each $T_i$.   Hence, each $T_i$ can be obtained via the order $v^1, v^k, v^2, \dots, v^{k-1}, v^{k+1}, \dots, v^n$. Since $T= \bigcup_{i=1}^{m} T_i$,  $T$ can be obtained via the  order $v^1, v^k, v^2, \dots, v^{k-1}, v^{k+1}, \dots, v^n$.  
\end{proof}

Note that Lemma~\ref{lemma-part1A} and Lemma~\ref{lemma-part1B} together let us move $v^k$ to the second position in the order (assuming  
pulling or pushing $v^1$ gives a nontrivial subdivision of $V$).    Also, in both cases, if $v^k$ was a candidate for pulling (pushing) first in $T$, when we move $v^k$ to the second position we obtain a  order for $T$ in which $v^k$  is a candidate for pulling (pushing) first in $T$ {\it and} $v^k$ is actually pulled (pushed) second in the order.

\section{Swapping the First Two Vertices}
\label{swap}

The remainder of  Theorem~\ref{thm-BIG} will be proved by looking at the four combinations of pulling and pushing the first two points $v^1$ and $v^2$.  In each case we will show the subdivision obtained by pulling/pushing $v^1$ and then pulling/pushing $v^2$ is the same subdivision obtained by switching the order of $v^1$ and $v^2$.  If these subdivisions  are the same, then the triangulations obtained by the orders $v^1,v^2,v^3, \dots, v^n$ and $v^2,v^1,v^3, \dots, v^n$ will be the same.

\begin{lemma}
Assume the hypotheses of Theorem~\ref{thm-BIG}.  Assume $v^1$ is pulled and $v^2$ is a candidate for pulling first in $T$. Then $T$ can be obtained via the order $v^2, v^1, v^3, \dots, v^n$, where $v^1$  and $v^2$ are pulled and $v^j, j > 2$, is pulled (pushed) if it was pulled (pushed) in the original order.
\label{lemma-BIGpartII1}
\end{lemma}

\begin{proof}
By the comments following Lemmas~\ref{lemma-part1A} and ~\ref{lemma-part1B}  we may assume $v^2$ is in fact pulled. 

Let $S_1$ be the subdivision obtained by pulling $v^1$ first.  Then $S_1$ is the collection of  pyramids   
$$S_1=\{F \cup \{v^1\}: \text{$F$ is a facet of $V$ not containing $v^1$} \}.$$ 
Since $v^2$ is a candidate for pulling first, it must be in every pyramid of $S_1$.  In particular, $v^2$ is in every facet $F$ of $V$ not containing $v^1$, so every facet of $V$ must contain at least one of $v^1$ and $v^2$.

Take $S_{12}$ to be the subdivision obtained by pulling $v^2$ in $S_1$.  Recall that the triangulation of a pyramid is induced by the triangulation of its
base.  Hence, each facet $F$ above is subdivided in $S_{12}$ into a collection of pyramids having apex $v^2$ and base $H$ a facet of $F$ (subfacet of $V$) not containing $v^2$.

Let $\mathcal{F}$ be the collection of such subfacets $H$: subfacets
of $V$ such that $v^1,v^2\not\in H$ and $H\subset F$ for some facet $F$ of $V$ with $v^1\not\in F$ and $v^2\in F$.  Let $H$ be any subset of $V$ such that $v^1,v^2\not\in H$.  Then, using the fact that every facet of $V$ must contain at least one of $v^1,v^2$,  for the two facets of $V$ containing $H$ it must be the case that exactly one contains $v^1$ but not $v^2$ and the other contains $v^2$ but not $v^1$.  We conclude
$\mathcal{F}=\{H:H$ is a subfacet of $V$ and $v^1,v^2\not\in H\}$
and $$S_{12}=\{H \cup \{v^1,v^2\}: H \in \mathcal{F} \}.$$
Thus the $d$-polytopes in $S_{12}$ are precisely the sets (two-fold pyramids) of the form $H\cup\{v^1,v^2\}$ for $H\in\mathcal{F}$.  

Now let $S_2$  be the subdivision obtained by pulling $v^2$ first,
$$S_2=\{G \cup \{v^2\}: \text{$G$ is a facet of $V$ not containing $v^2$} \},$$
and $S_{21}$ be the subdivision of $V$ obtained by pulling $v^1$ in $S_2$. Noting that every facet of $V$ not containing $v^2$ must contain $v^1$, we have facet $G$ of $S_2$  subdivided in $S_{21}$ into a collection of pyramids having apex $v^1$ and base $H$ a facet of $G$ (subfacet of $V$) not containing $v^1$.  These subfacets $H$ of $V$ are those satisfying $v^1,v^2\not\in H$ and $H\subset G$ for some facet $G$ of $V$ with $v^2\not\in G$ and $v^1\in G$.  By the same argument as before, this set of subfacets is $\mathcal{F}$ and $d$-polytopes in $S_{21}$ are precisely the sets (two-fold pyramids) of the form $H\cup\{v^1,v^2\}$ for $H\in\mathcal{F}$.  

Therefore $S_{12}=S_{21}$ and the result follows.  
\end{proof}


The next step is to show that if $T$ is a lexicographic triangulation of $V$ given by the order $v^1, \dots, v^n$ and  $v^1$ is pushed first, $v^2$ is pulled second, and $v^2$ is a candidate for pulling first, then we can switch the order of $v^1$ and $v^2$ and obtain the same triangulation $T$.  This case is similar to the case where $v^1$ is pulled first and $v^2$ is a candidate for pushing first.  Before we prove either case, we will examine the possible point sets whose convex hulls are $(d-1)$-dimensional. Some of these sets will be the facets of $V$, others will be sets that become important in the subdivisions obtained by pulling/pushing $v^1$ and $v^2$.

  Recall that we have already handled the case when pulling/pushing $v^1$ gives the trivial subdivision.  Thus, ${\mathrm{conv}}(V-\{v^1\})$ is $d$-dimensional, and its facets have dimension $d-1$.   We will denote the shadow of $v^1$ in $V-\{v^1\}$ as $S(v^1)$ and the shadow boundary of the shadow as $\partial S(v^1)$.  

 We examine the sets of dimension $d-1$ in four cases. 
\begin{enumerate}
\item  First, consider the facets $H$  of $V$ that contain the point $v^1$.  Either $H$ is a pyramid with apex $v^1$, or it is not.  If $H$ is a facet of $V$ that is a pyramid with apex $v^1$, then $H= G \cup \{v^1\}$, where $G$ is a $(d-2)$-dimensional set of the shadow boundary; we will write  $G \in \partial S(v^1)$.  We define two sets:
\begin{align*}
\mathcal{F}_1 &:= \{F= G \cup \{v^1\}: F \text{ is a facet of } V, ~G \in \partial S(v^1), ~v^2 \in G\}, \text{ and} \\
\mathcal{F}_2 &:= \{F= G \cup \{v^1\}:  F \text{ is a facet of } V, ~G \in \partial S(v^1), ~v^2 \notin G\}.
\end{align*}

Now suppose $H$ is a facet of $V$ containing $v^1$ that is not a pyramid with apex $v^1$.  We define $\mathcal{F}_3$ to be the set of all facets of $V$
containing both $v^1$ and $v^2$ that are not pyramids with apex $v^1$ and  $\mathcal{F}_4$ to be the set of all facets of $V$ containing $v^1$, but not containing $v^2$, that are not pyramids with apex $v^1$.  Note that these facets $H$ are precisely those that ``split'' when we push $v^1$ in the trivial subdivision.  That is, pushing $v^1$ subdivides $H$ into  two types of sets of dimension $d-1$: $H -\{v^1\}$, and  sets of the form $G \cup \{v^1\}$, where  $G \in \partial S(v^1)$.   We  define four new sets of dimension $d-1$:
\begin{align*} 
\mathcal{F}_5 &:= \{F-\{v^1\}: F \in \mathcal{F}_3 \}, \\
\mathcal{F}_6 &:= \{F-\{v^1\}: F \in \mathcal{F}_4\}, \\
\mathcal{F}_7 &:= \{F=G \cup \{v^1\}: F \in \mathcal{F}_3, ~G \in \partial S(v^1) \}, \\
\mathcal{F}_8 &:= \{F=G \cup \{v^1\} : F \in \mathcal{F}_4, ~G \in \partial S(v^1) \}.
\end{align*}

\item  The second type of $(d-1)$-dimensional sets are facets of $V$ not containing $v^1$.  Note that these facets will also be facets
of $V-\{v^1\}$.  We define two types of sets:
\begin{align*}
\mathcal{F}_9 &:= \{F: \text{$F$ is a facet of $V$}, ~v^1 \notin F, ~v^2 \in F\}, \text{and} \\
\mathcal{F}_{10}&:= \{F: \text{$F$ is a facet of $V$}, ~v^1 \notin F, ~v^2 \notin F\}.
\end{align*}

\item  The third type of $(d-1)$-dimensional sets are facets of $V-\{v^1\}$ that are not facets of $V$.  These are the facets of $V-\{v^1\}$ visible from $v^1$.  We define 
\begin{align*}
\mathcal{F}_{11}&:= \{F: \text{$F$ is a facet of $V-\{v^1\}$ in $S(v^1)$} \}.
\end{align*}
Note that the facets of $V-\{v^1\}$ are precisely the sets of $\mathcal{F}_9 \cup \mathcal{F}_{10}\cup\mathcal{F}_{11}$.
\item  Let $R$ be a $(d-2)$-dimensional face of $S(v^1)-\partial S(v^1)$; that is, $R$ is a facet of a facet in the shadow of $V-\{v^1\}$ visible from $v^1$, but it is not in the shadow boundary.  If $R=F' \cap F''$, where $F',F'' \in \mathcal{F}_{11}$, then $R \cup \{v^1\}$ is $(d-1)$-dimensional.  We define
\begin{align*} 
\mathcal{F}_{12}&:= \{R \cup \{v^1\}: \text{$R=F' \cap F''$, where $F',F'' \in \mathcal{F}_{11}$, $\dim({\mathrm{conv}}(R))=d-2$} \}.
\end{align*}
\end{enumerate}

We are now ready to settle the case where $v^1$ is pushed first, $v^2$ is pulled second, and $v^2$ is a candidate for pulling first.
\begin{lemma}
Assume the hypotheses of Theorem~\ref{thm-BIG}.  Suppose $v^1$ is pushed and $v^2$ is a candidate for pulling first in $T$. Then $T$ can be obtained via the order $v^2, v^1, v^3, \dots, v^n$, where $v^2$ is pulled, $v^1$ is pushed,  and $v^j, ~j > 2$, is pulled (pushed) if it was pulled (pushed) in the original  order.
\label{lemma-BIGpartII2}
\end{lemma}

\begin{proof}

The case when $v^1$ is not a vertex of $V$ follows easily, since then it is also not a vertex of the subdivision resulting from pulling $v^2$ first, and pushing it in either case simply removes it from all sets in the subdivision.

Let $W$ be the subdivision obtained from the trivial subdivision by pushing $v^1$.  Then  $$W=\{ V-\{v^1\}\}\cup\{F \cup \{v^1\}: F \in \mathcal{F}_{11} \}.$$
The $(d-1)$-faces of the sets of $W$ are exactly $\mathcal{F}_1 \cup \mathcal{F}_2 \cup \mathcal{F}_5 \cup \cdots \cup \mathcal{F}_{12}.$

Now, let $W'$ be the refinement of $W$ obtained by pulling $v^2$.  Since $v^2$ is a candidate for pulling first,  every $S_i \in  W$ must contain $v^2$. 
We examine all $(d-1)$-faces of the sets of $W$.  Since $v^2$ is in every set of $W$,  in particular we know $v^2 \in F$ for all $ F \in \mathcal{F}_{11}$.  A $(d-2)$-dimensional face $R$ of $S(v^1)-\partial S(v^1)$ is given by $R=F' \cap F''$, where $F',F'' \in \mathcal{F}_{11}$.  Thus, $v^2 \in R$ for every such $R$ and $v^2 \in F$ for all $ F \in \mathcal{F}_{12}$.   Therefore, if $F \in \mathcal{F}_1 \cup \mathcal{F}_5  \cup \mathcal{F}_7  \cup \mathcal{F}_9  \cup \mathcal{F}_{11}  \cup \mathcal{F}_{12}$, then $v^2 \in F$.   Note that by definition, if $F \in \mathcal{F}_2 \cup \mathcal{F}_6 \cup \mathcal{F}_8 \cup \mathcal{F}_{10}$, then $v^2 \notin F$.  Thus, $$W'=\{F \cup \{v^2\}: F \in \mathcal{F}_2 \cup \mathcal{F}_6 \cup \mathcal{F}_8 \cup \mathcal{F}_{10} \}.$$

Now, let $U$ be the refinement of the trivial subdivision obtained by pulling $v^2$.  Then 
\begin{align*}
U & = \{F \cup \{v^2\}: \text{$F$ is a facet of $V$ not containing $v^2$} \} \\
& = \{F \cup \{v^2\}: F \in \mathcal{F}_2 \cup \mathcal{F}_4\cup \mathcal{F}_{10} \}.
\end{align*}

Let $U'$ be the refinement of $U$ obtained by pushing $v^1$.  We want to determine the sets of $U'$ and show $U'=W'$.  So, $\{F \cup \{v^2\}: F \in \mathcal{F}_{10} \} \subseteq U'$ since $F \in \mathcal{F}_{10}$ implies $v^1 \notin F$.  Note that $F\cup \{v^2\}$ is a pyramid with apex $v^2$ and can therefore be subdivided by the induced subdivision of $F$.  If $F \in \mathcal{F}_{2}$, then $F=G \cup \{v^1\}$, where $v^2 \notin G$ and $G$ is a face of dimension $(d-2)$ of the shadow boundary.  Thus, $F$ is a pyramid with apex $v^1$.  Pushing the apex of a pyramid does not change the subdivision.  Thus, $F \cup \{v^2\} \in U'$ for $F \in \mathcal{F}_2$.  Suppose $F \in \mathcal{F}_4$.  Since $F$ is not a pyramid with apex $v^1$, pushing $v^1 \in F$ subdivides $F$ into $F-\{v^1\}$ and sets of the form $\{G \cup \{v^1\}: G \in \partial S(v^1) \}.$  Since $v^2 \notin F$, this subdivision is given precisely by the sets of $\mathcal{F}_6 \cup \mathcal{F}_8$.  So if $F \in \mathcal{F}_4$,  $F \cup \{v^2\}$ is subdivided into sets of the form  $\{ H\cup \{v^2 \}: H \in \mathcal{F}_6 \cup \mathcal{F}_8\}.$  It follows that $$U'=\{F \cup \{v^2\}: F \in \mathcal{F}_2 \cup \mathcal{F}_6 \cup \mathcal{F}_8\cup \mathcal{F}_{10} \}.$$  Thus, $W'=U'$, and we can obtain $T$ by switching the order of $v^1$ and $v^2$.  \end{proof}

We mentioned, when we defined the sets of dimension $d-1$, that the case of Lemma \ref{lemma-BIGpartII2} and the case where $v^1$ is pulled first, $v^2$ is pushed second, and $v^2$ is a candidate to be pushed first, are very similar.  Note that in the proof of Lemma~\ref{lemma-BIGpartII2}  the subdivision $U$ was obtained by pulling $v^2$, and $U'$ was a refinement of $U$ obtained by pushing $v^1$ in $U$.  Finding the sets of both subdivisions did not require using the fact that $v^2$ was a candidate to be pushed first.  Thus, this proof may be used to prove the following lemma.  Note that we have  assumed a different order for $T$ so that we may use the same description of the $\mathcal{F}_i$'s.

\begin{lemma}
Let $V \subset {\mathbb R}^d$ be a finite point set with $\dim ({\mathrm{conv}} (V))=d$.  Suppose $T$ is a lexicographic triangulation of $V$ given by the order $v^2, v^1, v^3, \dots, v^n$.  Assume $v^2$ is pulled first, $v^1$ is pushed second,  and $v^1$ is a candidate for pushing first in $T$. Then $T$ can be obtained via the order $v^1, v^2, v^3, \dots, v^n$, where $v^1$ is pushed, $v^2$ is pulled, and $v^j, j >2,$  is pulled (pushed) if it was pulled (pushed) in the original order.
\label{lemma-BIGpartII3}
\end{lemma}
 
\begin{proof}

The case when $v^1$ is not a vertex of the subdivision resulting from pulling $v^2$ first is easy, since then it is not a vertex of $V$, and pushing it in either case simply removes it from all sets in the subdivision.

Let $U$ be the subdivision obtained by pulling $v^2$, and let $U'$ be the subdivision obtained by refining $U$ by pushing $v^1$ as in the proof of Lemma~\ref{lemma-BIGpartII2}.  We must show  the subdivision obtained by pushing $v^1$ and then pulling $v^2$ is $U'$. Let $W$ be the subdivision obtained from the trivial subdivision by pushing $v^1$:
$$W=\{V-\{v^1\}\}\cup \{F \cup \{v^1\}: F \in \mathcal{F}_{11} \} .$$
We claim $T$ is a refinement of $W$.  Since $v^1$ is a candidate for pushing first in $T$, $z_1(T)={\mathrm{vol}}(C)$, where $C=\bigcup_{\{F \in \mathcal{F}_{11} \} } {\mathrm{conv}} ( F \cup \{v^1\}).$  (Note that $C$ may not be convex\@).  Thus every simplex in $T$ containing $v^1$ must lie in $C$.  

We claim that $v^2\in F$ for every $F\in\mathcal{F}_{11}$.  For suppose there is some $F\in\mathcal{F}_{11}$ such that $v^2\not\in F$. Let $H$ be the hyperplane containing $F$.  Then $v^1$ and $v^2$ lie in opposite open halfspaces of $H$.  There is a ray from $v^2$ intersecting the relative interior of $F$ and containing some point $w$ in the relative interior of $\conv(F')$ for some 
$F'\in \mathcal{F}_2\cup\mathcal{F}_8$. 
Hence there is some pyramid $P$ in $U$ with apex $v^2$ that contains $v^1$ (in its base), $v^2$, and $w$.  Since the line segment joining $v^2$ and $w$ does not lie entirely in $C$, neither does $P$.  Now pushing $v^1$ subdivides $P$ into $d$-polytopes, one of which, $P'$, contains $w$ and $v^2$ in its convex hull.  But $w\in F'$ implies that $w$ must also be in $P'$, which is a two-fold pyramid over $P'-\{v^1,v^2\}$.  Therefore 
the triangulation $T$ subdivides contains a $d$-simplex whose convex hull contains $v^1$, $v^2$ and $w$ and thus is not contained entirely in $C$.  This contradiction establishes the claim.

Now, let $W'$ be the refinement of $W$ obtained by pulling $v^2$.  Since $v^2$ is in every set of $W$, the remainder of this proof follows from the proof in  Lemma~\ref{lemma-BIGpartII2}.  Thus, $$W'=\{F \cup \{v^2\}: F \in \mathcal{F}_2 \cup \mathcal{F}_6 \cup \mathcal{F}_8 \cup \mathcal{F}_{10} \},$$  and $U'= W'$.  Thus, the order of the first two points   may be switched.
\end{proof}


We now turn to the last case of Theorem~\ref{thm-BIG} where $v^1$ and $v^2$ are pushed and $v^2$ is a candidate for pushing first.  We begin by defining a few types of facets.  First, let $V_1=V-\{v^1\}$, $V_2=V-\{v^2\}$, $V_{12}= V-\{v^1,v^2\}$, and $Q_i={\mathrm{conv}}(V_i)$, for $i=1,2,12$.

If $F$ is a facet of $V_{12}$ visible from $v^1$ (respectively, $v^2$) but not visible from $v^2$ (respectively, from $v^1$), then we will say $F$ is a {\it Type 1} (or {\it Type 2}) facet of $V_{12}$.  If $F$ is a facet of $V_{12}$ visible from both $v^1$ and $v^2$, then we will say  $F$ is a {\it Type 12} facet of $V_{12}$.  Note that a Type 1 (or Type 2) facet need not be a facet of $V_1$ ($V_2$), but  it  must be contained in a facet of $V_1$ ($V_2$).

Note that if $\dim (Q_{12})=d-1$, then the affine span of $Q_{12}$ is a hyperplane, which defines two open halfspaces. We will call these open half spaces the positive and negative ``sides'' of $V_{12}$.  We call  a ``side'' of $V_{12}$  a Type 1 (or Type 2, or Type 12) facet if $v^1$ (or $v^2$,  or both $v^1$ and $v^2$) is (is, are) in the corresponding half space.  In this way, if $v^1$ and $v^2$ lie in different half spaces we will have both Type 1 and Type 2 facets of $V_{12}$.  

The next lemma shows that in the final case of Theorem~\ref{thm-BIG} there are no Type $12$ facets of $V_{12}$.

\begin{lemma}
Suppose  $v^1$ is pushed first, $v^2$ is pushed second, and $v^2$ is also a candidate for pushing first.  Then there are no Type 12 facets of $V_{12}$.
\label{lemma-no-12-facets}
\end{lemma}

\begin{proof}
A Type 12 facet defines two closed half spaces: one contains $V_{12}$ but not $v^2$ and the other contains $v^2$.  Since a Type 12 facet is seen by $v^2$, $Q_1= {\mathrm{conv}}(V_1)= {\mathrm{conv}}(V_{12} \cup \{v^2\})$ will contain all Type 12 facets in its interior.  Hence, a Type 12 facet is not visible in $V_1$ from $v^1$.  Similarly, a Type 12 facet is not visible in $V_2$ from $v^2$.  Since $v^2$ is a candidate to be pushed first and Type 12 facets are not Type 2 facets, no triangulation of $V$ will have a simplex with vertex $v^2$ and base contained in a Type 12 facet.

Now, let $W$ be the subdivision obtained by refining the trivial subdivision by pushing $v^1$.  Then $V_1 \in W$.  Let $W'$ be the subdivision obtained by refining $W$ by pushing $v^2$.  Since $v^2 \in V_1$, $W'$ will contain sets of the form $F \cup \{v^2\}$, where $F$ is a facet of $V_1$ visible from $v^2$ but does not contain  $v^2$.  These facets are precisely the Type 2 and Type 12 facets of $V_{12}$.  Any triangulation obtained by refining $W'$ will therefore contain simplices having $v^2$ as a vertex and base contained in a Type 2 or a Type 12 facet,  a contradiction.  Thus, there are no Type 12 facets of $V_{12}$.  
\end{proof}

Note that by the way we defined the Type 1 and Type 2 facets of $V_{12}$ when $\dim(Q_{12})=d-1$, the argument of Lemma~\ref{lemma-no-12-facets} will still hold.  

We now define other types of facets obtained by pushing $v^1$ and $v^2$ in various cases.  Let $H$ be a facet of $V_1$ containing  $v^2$  and visible from $v^1$.  Push the point  $v^2$ in $H$ to obtain a subdivision of $H$.  Those sets in the subdivision that contain $v^2$ (and hence pyramids with apex $v^2$) are called {\it Type} $\mathcal{F}_1$ facets.  

If $H$ is a facet of $V_1$ containing $v^2$ that is not visible from $v^1$, and we push $v^2$ as above, we similarly obtain a subdivision of $H$.   In this case, we call the sets containing $v^2$  {\it Type $\overline{ \mathcal{F}_1} $} facets.  We define $\mathcal{F}_2$ and $\overline{ \mathcal{F}_2}$ similarly.  See Figure~\ref{figure-definition-facets} for examples in ${\mathbb R}^2$.
\begin{figure}
$$ \includegraphics[scale=0.8]{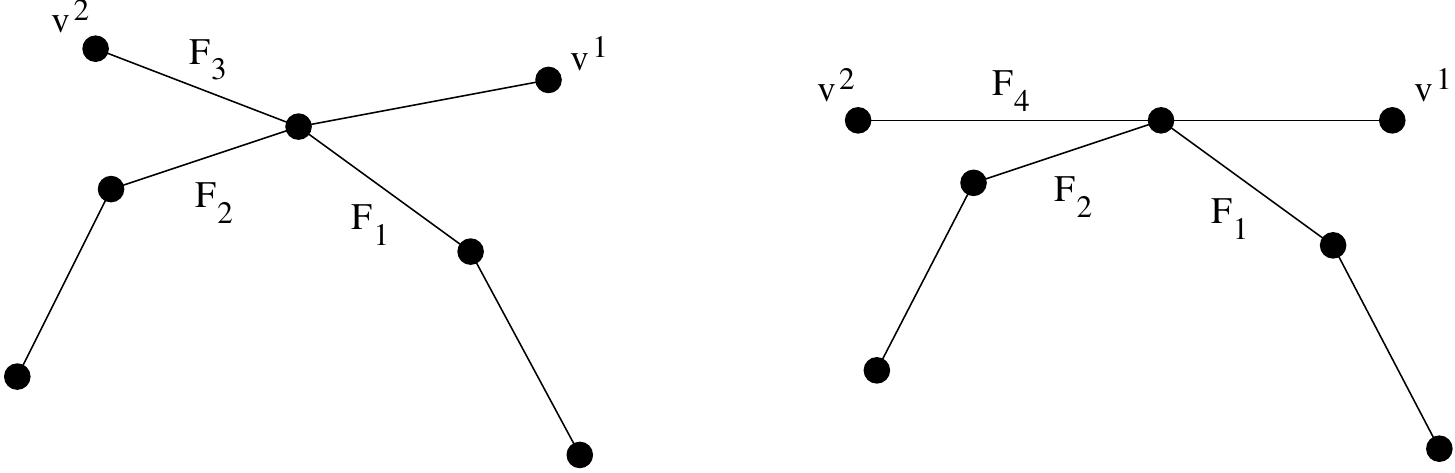}$$
\caption{Type 1, 2, $\mathcal{F}_1$, $\overline{\mathcal{F}_1}$ facets in ${\mathbb R}^2$.}
\label{figure-definition-facets}
\end{figure}
In these examples, $F_1$ is a Type 1 facet, $F_2$ is a Type 2 facet, $F_3$ is a Type $\mathcal{F}_1$ facet, $F_4$ is a Type $\overline{\mathcal{F}_1}$ facet.
 
Now, while there are no facets of $V_{12}$ visible from both $v^1$ and $v^2$, there may be a $(d-2)$-dimensional face $G$ seen by both points.  Such a $G$ is a $(d-2)$-dimensional face of the form  $F' \cap F''$, where $F'$ is a Type 1 facet and $F''$ is a Type 2 facet.    Then, in $V_1$, we have $G=F' \cap F_1$, where $F_1$ is a Type $\mathcal{F}_1$ or a Type $\overline{\mathcal{F}_1}$ facet (denoted by $F_1 \in \mathcal{F}_1 \cup \overline{\mathcal{F}_1}$).  In $V_2$, we have  $G=F'' \cap F_2$, where $F_2 \in \mathcal{F}_2 \cup \overline{\mathcal{F}_2}$.  Since $\dim({\mathrm{conv}}(G))=d-2$,  $F_1=G \cup \{v^2\}$ and  $F_2=G \cup \{v^1\}$.  We claim  $F_i \in \mathcal{F}_i$, $i=1,2$, or else $F_i \in \overline{\mathcal{F}_i}$, $i=1,2$.

\begin{lemma}
Suppose $G=F' \cap F_1$ and $G=F'' \cap F_2$ as defined above.  Then $F_1 \in \mathcal{F}_1$ and $F_2 \in \mathcal{F}_2$, or else $F_1 \in \overline{\mathcal{F}_1}$ and $F_2 \in \overline{\mathcal{F}_2}$.
\label{lemma-both-same-Type}
\end{lemma} 

\begin{proof}
We consider first the case when $\dim ({\mathrm{conv}}(V))=2$.  In this case, $G$ is a point given by $G=\{q\}= H_1 \cap H_2$, where $H_i$ is a Type $i$ facet, $i=1,2$ (Figure~\ref{figure-set-up-both-same-Type}).
\begin{figure}
$$\includegraphics[scale=0.8]{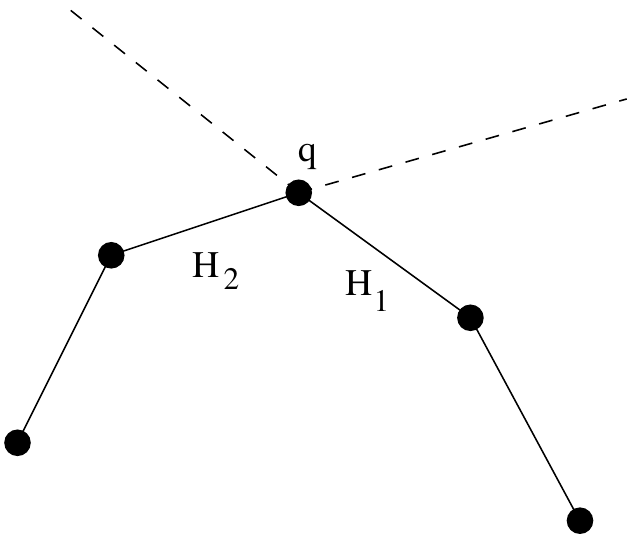}$$
\caption{The case where $\dim ({\mathrm{conv}}(V))=2$}
\label{figure-set-up-both-same-Type}
\end{figure}
Each facet $H_i$, $i=1,2$, defines two closed half spaces.  One of the half spaces will contain $V_{12}$ but not $v^i$.  Since there are no Type 12 facets, $v^1$ (respectively $v^2$)  must lie in the same closed half space defined by $H_2$ (respectively $H_1$) as $V_{12}$ does.  Note that by the definition of $G$ we have  $F_1 \cap H_1=\{q\}$, where $F_1=\{q,v^2\}$ in $V_1$ and $F_2 \cap H_2 =\{q\}$, where $F_2 =\{q,v^1\}$ in $V_2$.   There will be three cases to consider:
\begin{enumerate}
\item[(i)] ${\mathrm{conv}}(\{v^1,v^2\}) \cap Q_{12} = \emptyset$,
\item[(ii)] ${\mathrm{conv}}(\{v^1,v^2\}) \cap \partial Q_{12} \neq \emptyset$, but ${\mathrm{conv}}(\{v^1,v^2\}) \cap {\mathrm{int}}(Q_{12}) = \emptyset$, and
\item[(iii)] ${\mathrm{conv}}(\{v^1,v^2\}) \cap {\mathrm{int}}(Q_{12}) \neq \emptyset$.
\end{enumerate}

In (i) we have Figure~\ref{figure-final-case1}.  
\begin{figure}
$$\includegraphics[scale=0.8]{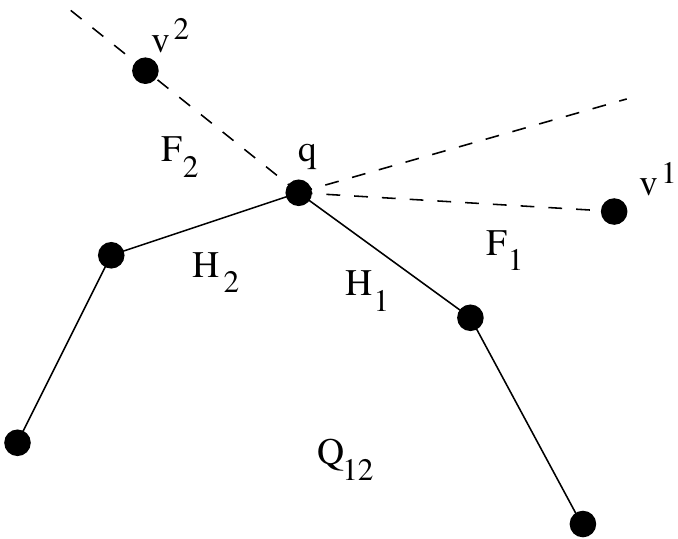}$$
\caption{Case (i), Lemma~\ref{lemma-both-same-Type}.}
\label{figure-final-case1}
\end{figure}
That is, the interior of the triangle given by $\{v^1,v^2,q\}$ does not intersect $Q_{12}$.  Thus, in this case $F_1 \in \mathcal{F}_1$ and $F_2 \in \mathcal{F}_2$.

Note that in (ii), if we have ${\mathrm{conv}}(\{v^1,v^2\}) \cap \partial Q_{12} \neq \emptyset$, but ${\mathrm{conv}}(\{v^1,v^2\}) \cap {\mathrm{int}}(Q_{12}) = \emptyset$, then $v^1, v^2$ are collinear with one or more points of $V_{12}$ in the boundary of $Q_{12}$ (Figure~\ref{figure-one-pt-collinear}).  Since $G \neq \emptyset$, there is exactly one point $q$, and we have the second diagram in Figure~\ref{figure-one-pt-collinear}.  
\begin{figure}
$$ \includegraphics[scale=0.8]{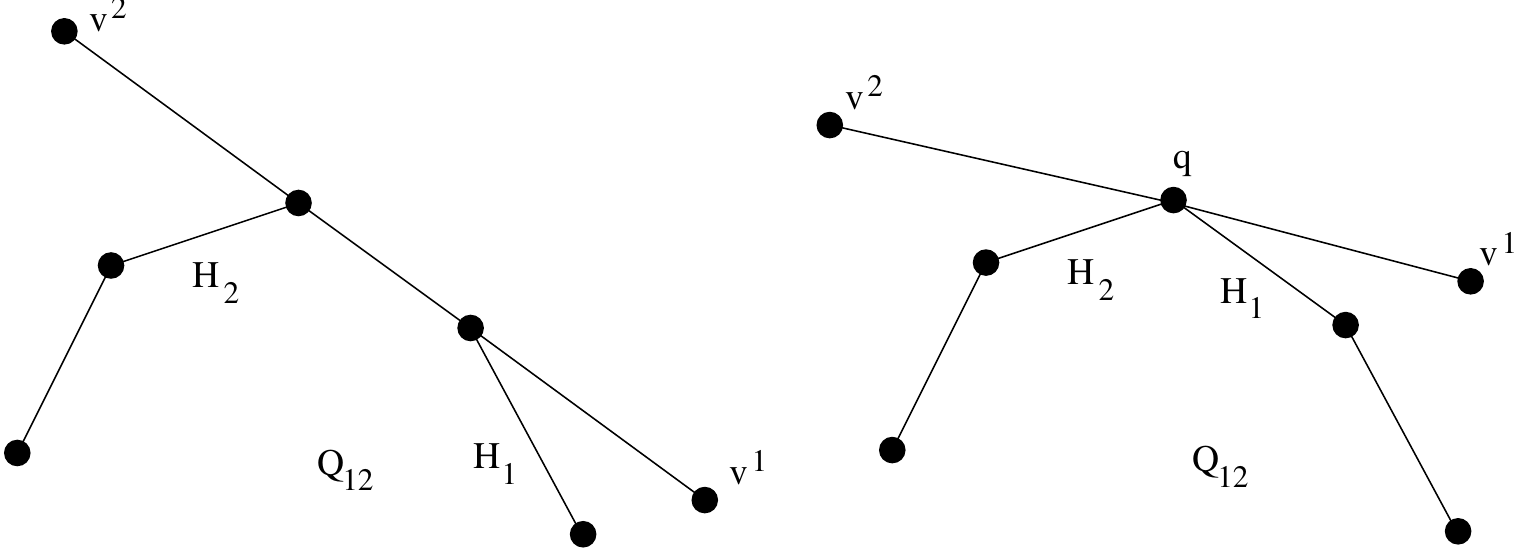}$$
\caption{Case (ii), Lemma~\ref{lemma-both-same-Type}.}
\label{figure-one-pt-collinear}
\end{figure}
In this case, any facet of $V_i$ containing $v^j$, $i \neq j$, is not visible from $v^i$, for $i,j \in \{1,2\}$.  Therefore,  if $F_i$ is a facet of $V_{i}$ containing $v^j$, $i \neq j$, then $F_i \in \overline{\mathcal{F}_i}$, for $i,j \in \{1,2\}$.

 Consider (iii). If ${\mathrm{conv}}(\{v^1,v^2\}) \cap {\mathrm{int}}(Q_{12}) \neq \emptyset$, then the interior of the triangle defined by $v^1,v^2,q$ intersects ${\mathrm{int}} (Q_{12})$ (Figure~\ref{figure-final-case3}).
\begin{figure}
$$\includegraphics[scale=0.8]{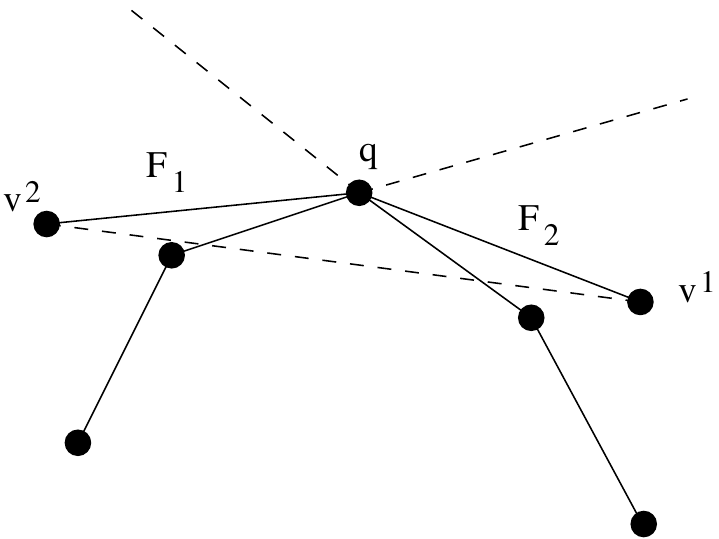}$$
\caption{Case (iii), Lemma~\ref{lemma-both-same-Type}.}
\label{figure-final-case3}
\end{figure}
In this case, the closed half space, $\mathcal{C}_i$, defined by ${\mathrm{aff}}(F_i)$ that contains $Q_{12}$, must also contain $v^i$, $i=1,2$. Since $q \in Q_{12}$, $\mathcal{C}_i$ will contain ${\mathrm{conv}}(\{q,v^i\})$, $i=1,2$.  But, $F_1= \{q,v^2\}$ and $F_2 = \{q,v^1\}$.   Thus, $F_i$ is not visible in $V_i$ from $v^i$, $i=1,2$.  Therefore $F_i \in \overline{\mathcal{F}_i}$, $i=1,2$.

These three cases prove the lemma for the case where $\dim ({\mathrm{conv}}(V))=2$.  Now consider the general case where $\dim({\mathrm{conv}}(V))=d$.  Recall that $G$ is a
$(d-2)$-dimensional face of $V_{12}$ given by $G=F' \cap F''$.  In $V_1$, we have $G=F' \cap F_1$, where $F_1 \in \mathcal{F}_1 \cup \overline{\mathcal{F}_1}$; in $V_2$, $G=F''\cap F_2$, where $F_2 \in \mathcal{F}_2\cup \overline{\mathcal{F}_2}$.

Let $\phi: {\mathbb R}^d \rightarrow {\mathbb R}^2$ be an orthogonal projection onto a plane perpendicular to ${\mathrm{aff}}(G)$ and defined by $\phi (G)=\{q\}$.  Then the affine spans of the $(d-1)$-dimensional $F_1, F_2, F', F''$ are mapped to lines in ${\mathbb R}^2$ and $\phi (V_{12})$ is a two-dimensional point set.

Since we have proved the lemma for $d=2$,  we have $\phi (F_i) \in \phi(\mathcal{F}_i)$ or  $\phi (F_i) \in \phi (\overline{\mathcal{F}_i})$, for $i=1,2$.  Since $\phi$ is an orthogonal projection, the necessary geometric properties of all sets will be preserved.  Thus, we will have $F_i \in\mathcal{F}_i$, for $i=1,2$, or $F_i \in\overline{\mathcal{F}_i}$, for $i=1,2$.  
\end{proof}

Thus, Lemma~\ref{lemma-both-same-Type} proves that if $G= F' \cap F''$ exists, then $F_1=G \cup \{v^2\}$ is an $\mathcal{F}_1$ facet if and only if $F_2 = G \cup \{v^1\}$ is an $\mathcal{F}_2$ facet.  In this case, we will have, by the definitions of the Type $\mathcal{F}_i$ facets, that $[G \cup \{v^2\}]\cup \{v^1\}$ is a pyramid of the subdivision obtained by pushing $v^1$ first and pushing $v^2$ second if and only if $[G \cup \{v^1\}]\cup \{v^2\}$ is a pyramid of the 
subdivision obtained by pushing $v^2$ first and pushing $v^1$ second. Note that since $G$ is $(d-2)$-dimensional, $G\cup \{v^1\}\cup \{v^2\}$ is a two-fold pyramid over $G$.

We are now ready to prove the last case of Theorem~\ref{thm-BIG}.

\begin{lemma}
Assume the hypotheses of Theorem~\ref{thm-BIG}.  Assume  $v^1$ is pushed first, $v^2$ is pushed second, and $v^2$ is a candidate for pushing first. Then $T$ is given by the order $v^2, v^1, v^3, \dots, v^n$, where $v^2$ is pushed first, $v^1$ is pushed  second, and $v^i$, $i \neq 1,2$, is pulled/pushed if it was pulled/pushed in the original order.
\label{lemma-BIGpartII4}
\end{lemma}

\begin{proof}
Let $W$ be the subdivision obtained by pushing $v^1$.  There are four types of sets in $W$:
\begin{enumerate}
\item[(i)] pyramids with apex $v^1$ over Type 1 facets,
\item[(ii)] pyramids with apex $v^1$ over Type $\mathcal{F}_1$ facets,
\item[(iii)] pyramids with apex $v^1$ over facets that are the union of the points in a Type 1 and a Type $\mathcal{F}_1$ facet, and
\item[(iv)] $V_1=V-\{v^1\}$.
\end{enumerate}

Let $W'$ be the refinement of $W$  obtained by pushing $v^2$.  The sets in $W$ given by (i) are sets in $W'$ since they do not contain $v^2$.  The sets in (ii) are two-fold pyramids over a $(d-2)$-dimensional face $G$ of $V_{12}$.  Since pushing the apex $v^2$ gives the same subdivision of these sets, they are also in $W'$.

The sets in (iii) are pyramids with $v^2$ in the base.  Thus, subdividing these can be done by subdividing the base $H$ by pushing $v^2$.  We defined $\mathcal{F}_1$ by this subdivision of $H$.  Hence, we obtain pyramids with apex $v^1$ over Type 1 and Type $\mathcal{F}_1$ facets.  The pyramids with apex $v^1$ over Type $\mathcal{F}_1$ facets are two-fold pyramids over a $(d-2)$-dimensional $G$.

The set in (iv) contains $v^2$.  Thus, pushing $v^2$ will give us the set $V_{12}=V_1-\{v^2\}$ and pyramids with apex $v^2$ over Type 2 facets.  

Hence, we have the following four types of sets in $W'$:
\begin{enumerate}
\item[(i)] pyramids with apex $v^1$ over Type 1 facets,
\item[(ii)] pyramids with apex $v^2$ over Type 2 facets,
\item[(iii)] two-fold pyramids over  $(d-2)$-dimensional faces $G$ of $V_{12}$ as described above, and
\item[(iv)] $V_{12}=V_1-\{v^2\}$.
\end{enumerate}

Note that the only place we have used the fact  $v^2$ is a candidate to be pushed first is Lemma~\ref{lemma-no-12-facets} where we proved  there are no Type 12 facets in $V_{12}$. Hence,   by a symmetric argument, the subdivision obtained by pushing $v^2$ first and pushing $v^1$ second is $W'$.  Thus, we can switch the order of $v^1$ and $v^2$ and obtain the same triangulation $T$.  
\end{proof}

This concludes the proof of Theorem~\ref{thm-BIG}.

\section{Recovering Lexicographic Triangulations}
\label{section-lexico-recover}
 In this section, we show  we can recover any lexicographic triangulation of $V \subseteq {\mathbb R}^d$ from its GKZ-vector using a greedy process.  We begin with a few definitions.

If $T$ is a lexicographic triangulation with order $v^1, \dots,v^n$, then pulling (pushing) $v^1$ gives the subdivision $S=\{S_1, \dots, S_m\}$ where each $S_i \in S$ is either a pyramid with apex $v^1$ or $S_i$ does not contain $v^1$.  When we pull/push $v^2$, we do so by pulling/pushing $v^2$ in each $S_i$. (Recall that if $v^2 \notin S_i$, then $S_i$ is a set in the refinement.)   We obtain $T$ by subdividing each $S_i$ according to the induced order and, therefore, obtain the induced triangulation $T_i$ of $S_i$.   Since each $S_i$ is subdivided by pulling/pushing $v^2$ first, $v^2$  is a candidate to be  pulled/pushed first in each induced triangulation, $T_i$.  Hence for the induced GKZ-vectors, $\phi(T_i)$, we have $z_2(T_i)=\min\{z_2(T_i'): \text{$T_i'$ is a triangulation of $S_i$} \}$ or $z_2(T_i)=\max\{z_2(T_i'): \text{$T_i'$ is a triangulation of $S_i$} \}$.  (Recall that if $v^2 \notin S_i$, then $z_2(T_i)=0$.)  Then, $z_2(T)=\sum_{i= 1}^{m} z_2(T_i').$  Hence, we will say  a point {\it $v^k$ is a candidate to be pulled second in $T$} if $$z_k(T)=\displaystyle \sum_{i= 1}^{m} \max\{z_k(T_i'): \text{$T_i'$ is a triangulation of $S_i$} \};$$  {\it $v^k$ is a candidate to be pushed second in $T$} if $$z_k(T)=\displaystyle \sum_{i= 1}^{m} \min \{z_k(T_i'): \text{$T_i'$ is a triangulation of $S_i$} \}.$$ In general,  if $T$ is a triangulation of $V \subseteq {\mathbb R}^d$ and $S=\{S_1, \dots, S_m\}$ is a subdivision of $V$ with $T \leq S$, then {\it $v^k$ is a candidate to be pulled next in $T$} if $$z_k(T)=\sum_{i=1}^{m} \max\{z_k(T_i'): \text{$T_i'$ is a triangulation of $S_i$} \}=\sum_{i: v^k k\in S_i} {\mathrm{vol}}(S_i),$$ and {\it $v^k$ is a candidate to be pushed next in $T$} if $$z_k(T)=\displaystyle \sum_{i= 1}^{m} \min \{z_k(T_i'): \text{$T_i'$ is a triangulation of $S_i$} \}=\sum_{i=1}^m [{\mathrm{vol}}(S_i)-{\mathrm{vol}}(S_i-\{v^k\})].$$

\begin{theorem}
Suppose $V=\{v^1, \dots, v^n\}  \subseteq {\mathbb R}^d$ and $T$ is a lexicographic triangulation of $V$.  Then $T$ can be recovered from its GKZ-vector.
\label{thm-lexico-recoverable}
\end{theorem}

\begin{proof}

Suppose $R=\{V\}$ is the trivial subdivision.  Clearly $T \leq R$. Since $T$ is lexicographic, there is a point $v^k$ that is a candidate to be pulled/pushed first in $T$.  By Theorem~\ref{thm-BIG}, we can actually pull/push $v^k$ first.  Without loss of generality, we assume $v^1$ is this point.  Let $R'=\{R_1, \dots, R_p\}$ be the subdivision obtained by refining the trivial subdivision by pulling/pushing $v^1$.  Each set $R_i$ is a pyramid with apex $v^1$ or it does not contain $v^1$.

Let $S=\{S_1, \dots, S_m\}$ be the lexicographic subdivision of $V$ obtained by pulling/pushing the points $v^1, \dots, v^k$ in that order.  Then $T \leq S$.
Let $T_i$ be the induced triangulation of  $S_i \in S$.  Suppose $v$ is a candidate to be pulled/pushed next in $S$.  Let $S'=\{S_i', \dots, S_t'\}$ be the subdivision obtained by pulling/pushing $v$ in $S$.  Then $S' \geq T$.  By definition, $v$ is  a candidate to be pulled/pushed first in each $T_i$.  By  Theorem~\ref{thm-BIG}, we can actually pull/push $v$ first in the order for $T_i$. Thus, $S_i$ can be refined to $T_i$ via the order whose first $k+1$ points
are $v, v^1, \dots, v^k$.  If  $v^j \in \{v^1, \dots, v^k\}$ is  present in $S_i$, then $S_i$ is  a pyramid with apex $v^j$.  Thus, we may pull/push $v^j$ at any point in the order for $T_i$.  So for each $S_i$, $T_i$ is a refinement of the subdivision obtained by pulling/pushing $v^1, \dots, v^k, v$ in that order.  Since $T$ is obtained by refining  each $S_i$, $T$ is a refinement of the lexicographic subdivision obtained by pulling/pushing $v^1, \dots, v^k,v$.  

Repeating this process for the remaining points in $V$ we obtain an order for $T$.  Consequently, we  recover $T$.

\end{proof}

We remark that to carry out our algorithm, we first need to be able to determine a candidate for pulling/pushing.  Determining these candidates involves determining maximum and minimum values for the induced GKZ-vector of each point in each set of the subdivision. We determine the maximum and minimum values by computing the volumes of convex hulls.  Algorithms for determining volumes  of convex hulls involve knowing the facets \cite{HRGZ}.  In fact, determining volumes can be done by means of a triangulation.  But, finding a triangulation involves determining the convex hull and hence the facets of a point set.

We also need to be able to carry out a subdivision, so we need to be able to perform a refinement by pulling/pushing a point in each set of the given subdivision.  Pulling/pushing a point $v^k$ will require being able to determine the facets of $V$ or $V-\{v^k\}$.

Now, at each step of the algorithm, we are calculating the volume of a pyramid or the volume of a point set with fewer points.    So, on one hand, things get a little easier at each step because we are either looking at the base of a pyramid or we are looking at smaller point sets.  But, at each step we also are computing the volume of more point sets.  It may be interesting to determine what this trade-off means and how efficiently  we can recover a lexicographic triangulation.

\end{sloppypar}




\end{document}